\documentclass[dvipsnames, a4paper, 11pt]{article}

\usepackage[utf8]{inputenc}
\usepackage[T1]{fontenc}
\usepackage{amsmath}
\usepackage{amsthm}
\usepackage{amsfonts}
\usepackage{amssymb}
\usepackage{graphicx,caption,subcaption}
\usepackage{fancybox}
\usepackage{enumitem}
\usepackage{soul}
\usepackage{mathtools}
\usepackage{lmodern}
\usepackage{stmaryrd}
\usepackage{multirow}
\usepackage{multicol}
\usepackage{xcolor}
\usepackage{tabularx}
\usepackage{array}
\usepackage{colortbl}
\usepackage{varwidth}
\PassOptionsToPackage{dvipsnames,svgnames}{xcolor}
\usepackage{authblk}
\usepackage{wrapfig}
\usepackage[explicit,calcwidth]{titlesec}
\usepackage{fancyhdr}
\usepackage[left=2.1cm, right=2.1cm, top=2cm, bottom=2cm]{geometry}
\usepackage{esint}
\usepackage{tocloft}
\usepackage{hyperref}
\usepackage{twoopt}
\hypersetup{
    colorlinks=true, %
    linktoc=all,     %
    linkcolor=blue,  %
}
\PassOptionsToPackage{unicode}{hyperref}
\PassOptionsToPackage{naturalnames}{hyperref}
\usepackage{float}

\labelformat{equation}{(#1)}
\labelformat{figure}{Figure #1}
\labelformat{section}{Section #1}

\usepackage[ruled,vlined,linesnumbered]{algorithm2e}
\labelformat{algocf}{Algorithm #1}

\usepackage{todonotes}
\setuptodonotes{fancyline, color=red!20, shadow, inline}

\usepackage{silence}

\usepackage[bbgreekl]{mathbbol}
\DeclareSymbolFont{bbold}{U}{bbold}{m}{n}
\DeclareSymbolFontAlphabet{\mathbbold}{bbold}
\DeclareSymbolFontAlphabet{\mathbb}{AMSb}%

\newtheorem{theorem}{Theorem}

\newtheorem{lemma}{Lemma}
\newtheorem*{assumption}{Assumption}

\newcommand{\R}[0]{\mathbb{R}}

\renewcommand{\P}[0]{\mathbb{P}}

\newcommand{\Sum}[2]{\displaystyle\sum\limits_{#1}^{#2}}

\newcommand{\Inter}[2]{\displaystyle\bigcap\limits_{#1}^{#2}}
\newcommand{\Reu}[2]{\displaystyle\bigcup\limits_{#1}^{#2}}
\newcommand{\dd}[0]{\mathrm{d}}

\newcommand{\blue}[1]{#1}%

\renewcommand{\Im}[0]{\mathrm{Im}}
\newcommand{\Ker}[0]{\mathrm{Ker}}

\renewcommand{\epsilon}{\varepsilon}

\DeclareMathOperator{\argmin}{argmin}

\newcommand{\W}[0]{\mathrm{W}}
\renewcommand{\SS}{\mathbb{S}}

\newcommand{\eqlabel}[1]{\tag{#1}\label{eqn:#1}}

\newcommand{\Bell}{{\boldsymbol\ell}}

\newcommand{\npoints}{n}
\newcommand{\nproj}{r}

\newcommand{\SW}{\mathrm{SW}}

\title{Reconstructing discrete measures from 
  projections. Consequences on the empirical Sliced Wasserstein Distance.} 
\author[1]{Eloi Tanguy}
\author[2]{R\'emi Flamary}
\author[1]{Julie Delon}
\affil[1]{Universit\'e Paris Cit\'e, CNRS, MAP5, F-75006 Paris, France}
\affil[2]{CMAP, CNRS, Ecole Polytechnique, Institut Polytechnique de Paris}
\date{April 2023}

\begin{document}

\maketitle

\begin{abstract} 
   This paper deals with the reconstruction of a discrete measure $\gamma_Z$ on $\mathbb{R}^d$ from the knowledge of
   its pushforward measures $P_i\#\gamma_Z$ by linear
   applications $P_i: \mathbb{R}^d \rightarrow \mathbb{R}^{d_i}$
   (for instance projections onto
   subspaces). The measure $\gamma_Z$ being fixed, assuming that the rows of the matrices $P_i$ are
   independent realizations of laws which do not give mass 
   to hyperplanes, we show that if $\sum_i d_i > d$, this 
   reconstruction problem has almost certainly a unique solution.
   This holds for any number of points in $\gamma_Z$. A direct
   consequence of this result is an almost-sure separability property on
   the empirical Sliced Wasserstein distance. 
\end{abstract}

\section{Introduction}

In this note, we are interested in the following question: for a given discrete probability measure $\gamma_Z $ on $\R^d$, and $\nproj$ linear transformations $P_i:  \R^d \rightarrow \R^{d_i}$, can we characterize the set of probability measures on $\R^d$ with exactly the same images as $\gamma_Z$ through all of the maps $P_i$? Formally, this set writes
\begin{equation}\eqlabel{RP}
	\mathcal{S} = \left\lbrace\gamma \in \mathcal{P}(\R^d)\ |\ \forall i \in \llbracket 1, \nproj \rrbracket,\; P_i\#\gamma = P_i\#\gamma_Z \right\rbrace,
\end{equation}
where $P_i\#\gamma$ denotes the push-forward of $\gamma$ by $P_i$, {\em i.e.} the measure on $\R^{d_i}$ such that for any Borelian $A \subset\R^{d_i} $, $(P_i\#\gamma)(A) =\gamma(P_i^{-1}(A))$, and $\mathcal{P}(\R^d)$ denotes the space of {probability} measures on $\R^d$. 
The set $\mathcal{S}$ is nonempty since it contains at least $\gamma_Z$. A natural underlying question is to know when we get uniqueness,   {\em i.e.}  when $\mathcal{S} = \{\gamma_Z\}$. Indeed, in this case $\gamma_Z$ can be exactly reconstructed from the knowledge of all the $ P_i \# \gamma_Z$, which is why we refer to this problem as a reconstruction problem.

This reconstruction problem appears in many applied fields where a multidimensional measure is known only through a finite set of images or projections. This is the case for instance in medical or geophysical imaging problems such as tomography~\cite{gardner1999uniqueness}. It is also strongly related to the separability properties of the empirical version of the  Sliced Wasserstein distance~\cite{Rabin_texture_mixing_sw, bonneel2015sliced}, which is frequently used in machine learning applications~\cite{karras2018progressive, deshpande_generative_sw, Wu2019_SWAE}. 

In our reconstruction problem, it is clear that if one of the $P_i$ is injective
(which implies $d \leq d_i$), then $\mathcal{S} = \{\gamma_Z\}$, which is why we focus here on the cases where none of the $P_i$ is injective. We will also assume in this note that the $P_i$ are surjective and that  all the $d_i$ are strictly smaller than $d$, since  we can always replace $\R^{d_i}$ by the smaller subspace $\mathrm{Im}(P_i)$.  
To the best of our knowledge, this problem has not been widely discussed in the literature, perhaps because of its apparent simplicity.   
 A close and more discussed question is the one of the existence of probability measures $\gamma$ with marginal constraints~\cite{kazi2019construction, dall2012advances, kellerer1964verteilungsfunktionen}.
 Existence results for such couplings are known for some families of measures~\cite{joe1993parametric}, or measures exhibiting some specific correlation structures~\cite{cuadras1992probability}. However, in the general case, even if marginal constraints are compatible with each other, the existence of solutions is not always ensured~\cite{gladkov2019multistochastic}. 
 
 Our study case is different, since the constraints are all obtained as push-forwards of an unknown $\gamma_Z$, and the central question is not existence but uniqueness of solutions.  It is well known that a measure is uniquely determined by its projections on \textit{all} lines of $\R^d$ (Cramer-Wold theorem~\cite{cramer1936some}), and more generally by its projections on a set of subspaces as soon as they cover the whole space together~\cite{renyi1952projections}.
\blue{The problem for a non-discrete target measure has been studied by Chafai \cite{chafai_projections} which considers a similar reconstruction problem from the viewpoint of characteristic functions, using the equivalence $\forall i \in \llbracket 1, \nproj \rrbracket,\; P_i\#\mu = P_i\#\nu = 0 \Longleftrightarrow \forall i \in \llbracket 1, \nproj \rrbracket,\; \forall x \in \R^d,\; \phi_\mu(P_ix) = \phi_\nu(P_ix)$, where $\phi_\mu$ and $\phi_\nu$ denote the respective characteristic functions of $\mu$ and $\nu$. In general, the knowledge of the characteristic function on a finite union of strict subspaces is insufficient to determine a measure \cite{chafai_projections}.}
For a finite number of directions and in the case of a discrete measure $\gamma_Z$, simple linear algebra shows that if the number $\nproj$ of projections is large enough, we get $\mathcal{S} = \{\gamma_Z\}$. When the $P_i$ are projections on different hyperplanes for instance, Heppes showed in 1956~\cite{heppes1956determination} that a discrete distribution of at most $\npoints$ points $\gamma_Z = \frac 1 \npoints \sum_{\ell=1}^\npoints \delta_{z_\ell} $ is uniquely characterized by its projections $P_i\#\gamma_Z$ if the number $\nproj$ of these projections is larger than $\npoints+1$, and that simple counter-examples could be exhibited with only $\nproj=\npoints$ hyperplanes.  
More recent works~\cite{reconstruct_projection} show that uniqueness can be ensured with less projections as soon as the set of points is known to belong to a specific quadratic manifold.  These results are deterministic, they hold for every set of points and hyperplanes with the appropriate cardinality. %
In this paper, we add some stochasticity to the problem, and  
assume that the lines of the matrices $P_i$~\footnote{With a slight abuse of notation, we use the same notation here for the linear maps and their associated matrices.} are i.i.d. following a law $\mathbb{P}$ which does not give mass to hyperplanes. Under this assumption, we show that if $\sum_{i=1}^\nproj d_i > d$,  then $\mathbb{P}$-almost surely $\mathcal{S} = \{\gamma_Z\}$. This result is very different from the ones already present in the literature: it holds only a.s., but this permits a considerably weaker condition on the $P_i$, and the condition for the reconstruction surprisingly does not depend on the number of points.

\section{Solutions of the Reconstruction Problem}

In this section, we characterize the set $\mathcal{S}$ of solutions
defined in~\ref{eqn:RP} depending on the set of linear maps $P_i$. We
write {$\gamma_Z = \sum_{\ell=1}^\npoints b_\ell \delta_{z_\ell} $}
with $Z = (z_1,\cdots,z_\npoints) \in (\R^d)^{\npoints}$, and assume
that all points are distinct ($k\neq j \Longrightarrow z_k \neq
z_j$). We also always assume that $\npoints >1$. {The weights $(b_\ell) \in (\R_{+}^*)^\npoints$ sum to one and are each nonzero.}

 As we shall see, given a discrete measure $\gamma_Z$ with $\npoints$
 points in dimension $d$, the Reconstruction Problem~\ref{eqn:RP}  has a unique solution $\mathcal{S}=\{\gamma_Z\}$ almost-surely when drawing the linear maps $P_i$ randomly, and when the dimensions strictly exceed $d$, {\em i.e.} when $D:=d_1 + \cdots + d_\nproj > d$.

\subsection{Computing Linear Push-Forwards of Discrete Measures}

Characterizing $\mathcal{S}$ requires the following technical Lemma, which provides a geometrical viewpoint of the push-forward operation.

\begin{lemma}[Linear push-forward formula]\label{lemma:linear_push_forward}
	Let $P \in \mathcal{M}_{h, d}(\R)$ of rank $h \leq d$ and $B \subset \R^h$.

	Then $P^{-1}(B) = P^T (PP^T)^{-1}B + \Ker P$. %

\end{lemma}

\ref{fig:linear_push_forward} shows a visualization of the set $P^T (PP^T)^{-1}B
+ \Ker P$, first where $B$ is comprised of two points of $\R^2$ and $\Ker P$ is
a horizontal plane in 3D, and second with $B$ a measurable set of $\R^2$. This illustrates the ill-posedness of the problem when the dimension of the {projections and number of projections is too small. In this case with $\nproj=1,\; d=3$ and $d_1 = 2$}, the condition $P^{-1}(A) = P^{-1}(B)$ leaves a degree of freedom, which we can visualize as the vertical axis here.

\begin{figure}[H]
	\centering
	\begin{subfigure}[c]{0.38\textwidth}
		\centering
		\includegraphics[width=\linewidth]{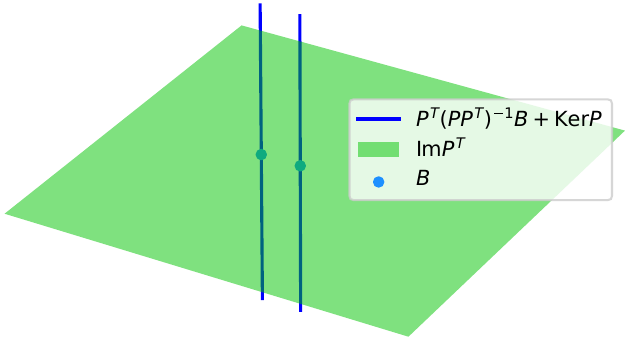}
	\end{subfigure}
	\hfill
	\begin{subfigure}[c]{0.6\textwidth}
		\centering
		\includegraphics[width=\linewidth]{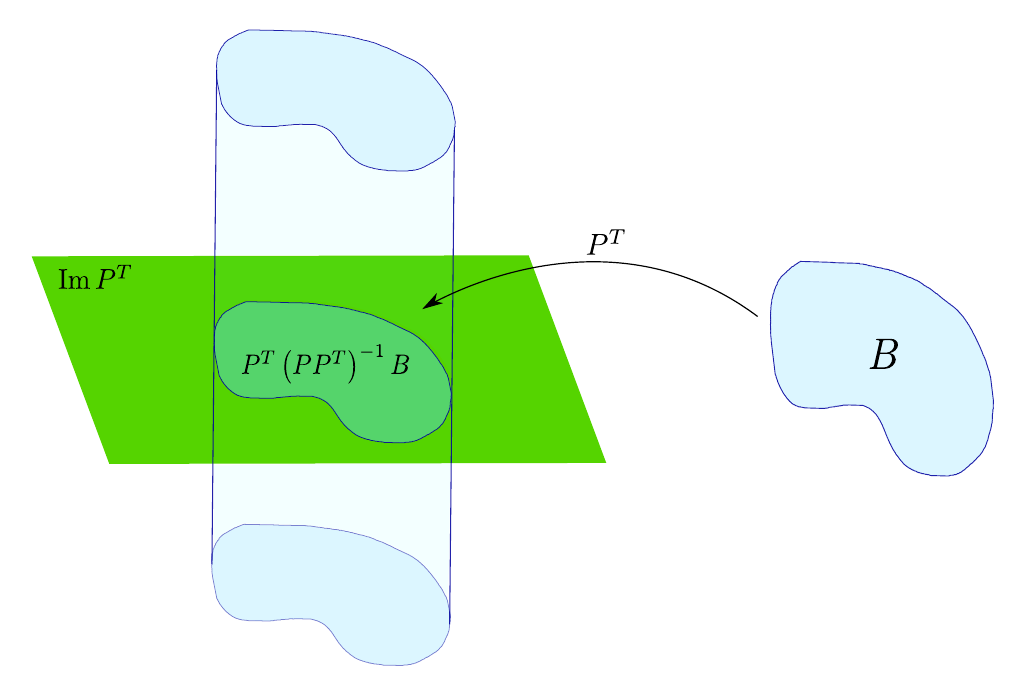}
	\end{subfigure}
	\caption{Illustrations of the linear push-forward formula $P^{-1}(B)$ for a
	3D to 2D projection $P$, (left) when $B$ is a set of two points and
	(right) for a more general set $B$.}
	\label{fig:linear_push_forward}
	\hfill
\end{figure}

\begin{proof}
	If $a \in P^T (PP^T)^{-1}B + \Ker P$, then by writing $a  = P^T (PP^T)^{-1}b + x$ with $b \in B$ and $x \in \Ker P$, we have $Pa = b \in B$, thus $a \in P^{-1}(B)$.

	For the opposite inclusion, consider $a \in P^{-1}(B)$. Since $P$ is of full rank $h$, we have the decomposition $\R^d = \Im P^T \overset{\perp}{\bigoplus} \Ker P$, with $Q := P^T (PP^T)^{-1}P$ the orthogonal projection on $\Im P^T$.

	Thus we can write $a = Qa + (I-Q)a = P^T (PP^T)^{-1}Pa + (I-Q)a $. Since $Pa \in B$, we conclude that $a \in P^T (PP^T)^{-1}B + \Ker P$.
\end{proof}

\subsection{Restraining the support of solutions of RP}

The following theorem states that the support of any solution of \ref{eqn:RP} is constrained to a  set $S$ obtained as the intersection of all sets $Z+\Ker P_i$. Without loss of generality, we will assume that each $P_i$ is of full rank $d_i$.

\begin{theorem}[Support of solutions of \ref{eqn:RP}]\

	If $\gamma$ is a solution of \ref{eqn:RP}, then 	 $\gamma\left(S \right) = 1$ with
		\begin{equation}\label{eqn:supp_RP2}
		S := \Inter{i=1}{\nproj} \left(Z + \Ker P_i\right) = \Reu{\substack{(\ell_1 ,\cdots, \ell_\nproj) \in \llbracket 1, \npoints \rrbracket^\nproj}}{}\quad\Inter{i=1}{\nproj}\;\left(z_{\ell_i} + \Ker P_i\right).
	\end{equation}

\end{theorem}

\begin{proof}
Using the same notations as in the proof of Lemma~\ref{lemma:linear_push_forward}, we write $Q_i := P_i^T (P_iP_i^T)^{-1}P_i$ the orthogonal projection on $\Im P_i^T$ and we recall the  decomposition $\R^d = \Im P_i^T \overset{\perp}{\bigoplus} \Ker P_i$.
Thus,  for  any borelian $A$ of $\R^d$, $A \subset Q_i A +  \Ker P_i$. Then $\gamma(A) \leq \gamma (Q_i A +  \Ker P_i) = P_i \# \gamma(P_iA)$, where the last equality is a direct consequence of Lemma~\ref{lemma:linear_push_forward}. 

Now, assume that $\gamma\in \mathcal{S}$ and define $S :=\Inter{i=1}{\nproj} K_i $ with  $K_i = Z+ \Ker P_i$. For each $i$, applying the previous inequality to $K_i^c$ yields $\gamma(K_i^c) \leq  P_i \# \gamma(P_iK_i^c)$. Since $\gamma$ is a solution, we have $P_i\#\gamma = P_i\#\gamma_Z$. By construction, $K_i = \{x,\;P_ix \in P_iZ\}$ thus  $K_i^c = \{x,\;P_ix \notin P_iZ\}$. Since $P_i\#\gamma$ is supported by $ P_iZ$, it follows that $P_i \# \gamma(P_iK_i^c) = 0$.
Finally, 
$\gamma(S^c) = \gamma(\Reu{i=1}{\nproj} K_i^c) \leq \sum_{i=1}^\nproj \gamma(K_i^c) =0$ and thus $\gamma(S)=1$.

\end{proof}

\ref{fig:reconstruction_grid} illustrates the previous result, with $\nproj=2$
projections onto lines in $\R^2$, with $\npoints=3$ points $Z = (z_1, z_2, z_3)$.
The support of any solution is confined to the intersections between any two
lines of the form $z_\ell + \Ker P_i$. Here this corresponds to the intersecting
points between an orange and a red line, allowing for 9 possible points,
including the original 3. In this case any weighting of the 9 Dirac masses that
respect the marginal constraints will give a solution: there exists an infinity
of possible solutions.

\begin{figure}[H]
	\centering
	\begin{subfigure}[c]{0.3\textwidth}
		\centering
		\includegraphics[width=0.7\linewidth]{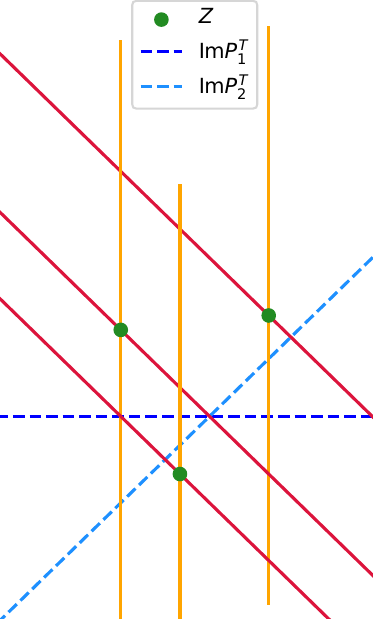}
	\end{subfigure}
	\hspace{50pt}
	\begin{subfigure}[c]{0.3\textwidth}
		\centering
		\includegraphics[width=0.7\linewidth]{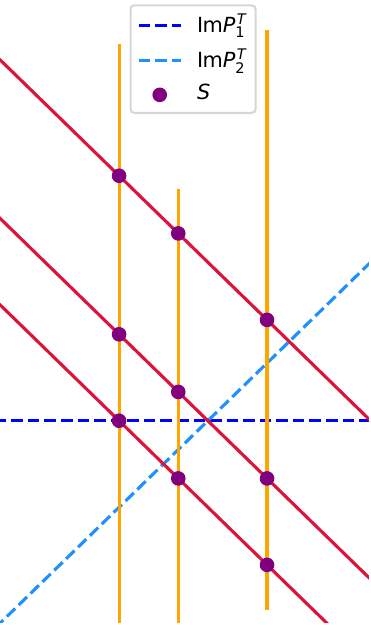}
	\end{subfigure}
	\caption{Illustration of the possible points for the support of a solution. On the left, $Z$ is the original measure points, and on the right, $S$ is the set of possible points for the support of a solution.}
	\label{fig:reconstruction_grid}
\end{figure}

\subsection{Conditions for unicity of solutions of RP}

Leveraging the previous support restriction and elementary random affine geometry, we can further restrict the condition on the set of solutions  $\mathcal{S}$. Theorem \ref{thm:as_unicity_reconstruction} below shows that if the random linear maps $P_i$ cover the original space $\R^d$ with redundancy (i.e. the sum of their target space dimensions strictly exceeds $d$), then almost surely, the reconstruction problem has a unique solution $\gamma_Z$. We formalize this random setting by the following assumption.

\begin{assumption}[$\mathcal{A}_\P$]
$$ \quad \forall i \in \llbracket 1, \nproj \rrbracket, \; P_i = \left(\begin{array}{ccc}
	\textrm{---} & \left(u_i^{(1)}\right)^T & \textrm{---} \\
	\textrm{---} & \vdots & \textrm{---} \\
	\textrm{---} & \left(u_i^{(d_i)}\right)^T & \textrm{---}
\end{array}\right)\quad  \text{ where } u_i^{(j)} \sim \P\ \mathrm{i.i.d,}$$
where $\P$ is a probability distribution over $\R^d$ s.t. for any hyperplane $H \subset \R^d,\; \P(H) = 0$.
\end{assumption}

The condition on the probabilities is verified in particular if $\P$ is absolutely continuous w.r.t. the Lebesgue measure of $\R^d$, or w.r.t. $\bbsigma$, the uniform measure over $\SS^d$ (the unit sphere of $\R^d$). These two examples are the most common for practical reconstruction problems, which is why we formulate $\mathcal{A}_\P$ in this manner.

The next theorems use assumption $\mathcal{A}_\P$ but still hold true
under milder hypotheses, where the lines $\left(u_i^{(j)}\right)^T$ of the matrices $P_i$
are assumed independent with (possibly different) probability laws giving no mass to hyperplanes.  

\begin{theorem}[Almost-sure unicity in
  \ref{eqn:RP}]\label{thm:as_unicity_reconstruction}
  Let $\gamma_Z$ be a fixed discrete probability measure. Assume that
  the matrices $P_i$ follow assumption $\mathcal{A}_\P$, and that  $D:=\Sum{i=1}{\nproj}d_i > d$. Then 
	$\P$-almost surely, $S = Z$ and  $\mathcal{S} = \{\gamma_Z\}$.
\end{theorem}

The idea behind the proof of Theorem~\ref{thm:as_unicity_reconstruction} is that $S$ is the union of sets of the form $\Inter{i=1}{\nproj}\;\left(z_{\ell_i} + \Ker P_i\right)$, which can be rewritten as  intersections of more than $d$ affine subspaces  in dimension $d$, thus are  $\P$-almost surely either singletons or empty. %

\begin{proof}
    \textrm{---} \textit{Step 1}: $S = Z$
	
	Let $\Bell := (\ell_1, \cdots, \ell_\nproj) \in \llbracket 1, \npoints \rrbracket^\nproj$ and $S_\Bell := \Inter{i=1}{\nproj}\;\left(z_{\ell_i} + \Ker P_i\right)$. We want to show $S_\Bell \subset Z$.
	
	First, observe that $x \in S_\Bell \Longleftrightarrow  \forall i \in  \llbracket 1, \nproj\rrbracket,\; \forall j \in  \llbracket 1, d_i\rrbracket,\;\;\; (u_i^{(j)})^T x = (u_i^{(j)})^T z_{\ell_i}$.

	We write $D = \sum_{i=1}^\nproj d_i$. %
	For the sake of simplicity, we rewrite the $k^{th}$ vector $u_i^{(j)}$ as $v_k$, where $k  \in  \llbracket 1, D\rrbracket$, and in the same way we write  $(w_k)_{k=1\cdots D}$ the vectors $(z_{\ell_1},\cdots, z_{\ell_1}, z_{\ell_2},\cdots, z_{\ell_2},\cdots, z_{\ell_\nproj},\cdots, z_{\ell_\nproj})$ with each $z_{\ell_i}$ repeated $d_i$ times. With these notations, 	we get
        \begin{equation}
          x \in S_\Bell \Longleftrightarrow v_k^T x = v_k ^T w_k,\;\; \forall k  \in  \llbracket 1, D\rrbracket.\label{linear_system}
        \end{equation}
        Let us call (LS) the linear system on the right of~\ref{linear_system}.
    (LS) has $D$ equations and $d$ unknowns, with $D > d$, it is therefore overdetermined.
    When all $w_k$ are equal, {\em i.e.} when $\Bell := (\ell, \cdots, \ell)$, clearly $x = z_\ell$ is a solution, which shows that $z_\ell \in S$ and thus $Z \subset S$.          

    If $\mathcal{A}_\P$ is satisfied, the matrix $U^{(d)} =
    \left(v_1,\cdots, v_d\right)^T $ is almost surely of full rank and
    the linear system   $v_k^T x = v_k ^T w_k$ for $k\in  \llbracket
    1, d\rrbracket$ almost surely has a unique solution $x^*$. The
    $(d+1)^{\mathrm{th}}$ equality of (LS) is $v_{d+1}^T (x^*-
    w_{d+1})=0$, which happens iif  $x^*= w_{d+1}$, or $x^*\neq
    w_{d+1}$ and $v_{d+1} \in  (x^*- w_{d+1})^{\perp}$. In the first
    case, the solution $x^*$ belongs to $Z$ since $w_{d+1}$ is one of
    the $z_{\ell_i}$. If $x^*\neq w_{d+1}$, since all the $\{v_k\}$ are
    i.i.d. of law $\P$, conditionally to $U^{(d)}$ the probability
    that $v_{d+1}$ is orthogonal to $(x^*- w_{d+1})$  is null and (LS)
    has almost surely no solution. We conclude that $S = Z$ almost surely.
    
    \textrm{---} \textit{Step 2}: The set of solutions of \ref{eqn:RP} is $\lbrace \gamma_Z \rbrace$ a.s.
    
    We have proven that $S = Z$ a.s., and thus that any solution $\gamma \in \mathcal{S}$ is supported by $Z$ a.s.. Let us write $\gamma = \Sum{\ell=1}{\npoints}a_\ell\delta_{z_\ell}$ and $\gamma_Z = \Sum{\ell=1}{\npoints}b_\ell\delta_{z_\ell}$. It follows in particular that $\Sum{\ell=1}{\npoints} (a_\ell - b_\ell) \delta_{P_1 z_\ell} = 0$, {and since for $k \neq \ell,\; z_k\neq z_\ell$,} hence $\mathbb{P}(P_1 z_\ell = P_1 z_k) \leq \mathbb{P}(u_1^{(1)} \in (z_\ell - z_k)^{\perp})=0$, thus $\forall \ell$, $a_\ell = b_\ell$ a.s..

\end{proof}

The previous Theorem~\ref{thm:as_unicity_reconstruction} only holds almost-surely, however "improbable" counter examples do exist with excessive symmetry. Below we present a counter-example adapted from \cite{reconstruct_projection}. Let $d := 2,\; \nproj := \npoints > d$ and $\forall i \in \llbracket 1, \nproj \rrbracket,\; d_i := 1$.

Consider $z_\ell := \left(\cos\left(\frac{(2\ell + 1)\pi}{\npoints}\right), \sin\left(\frac{(2\ell + 1)\pi}{\npoints}\right)\right)^T,\;P_\ell := \left(\cos\left(\frac{(2\ell + 1)\pi}{2\npoints}\right), \sin\left(\frac{(2\ell + 1)\pi}{2\npoints}\right)\right).$

As can be seen below (\ref{fig:CE_reconstruction}), for $\npoints=3$, this corresponds to placing the $(z_\ell)$ on every other vertex of a regular $2\npoints$-gon, and defining the $P_\ell$ such that $\Im P_\ell^T$ is the $\ell$-th bisector of the $2\npoints$-gon.

\begin{figure}[H]
	\centering
	\includegraphics[width=.6\linewidth]{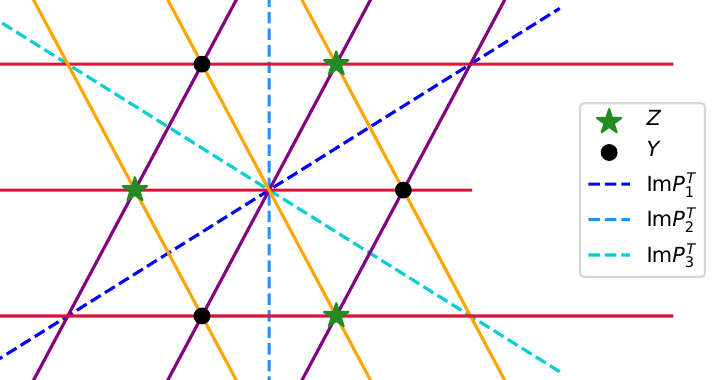}
	\caption{Illustration of a pathological super-critical case without unicity
	for specific projections $P_i$. In this case, $Y$ and $Z$ are distinct
	solutions with the same projections.}
	\label{fig:CE_reconstruction}
\end{figure}

The points of $S$ are the points of the form $\Inter{i=1}{3} \left(z_{\ell_i} + \Ker P_i\right)$, or visually the intersection points of a yellow line, a red line and a purple line. We can see that the remaining vertices of the polygon constitute another valid measure $\gamma_Y$ whose push-forwards $P_i\#\gamma_Y$ are all the same as those of the original measure.

As mentioned in~\cite{reconstruct_projection}, for a fixed list of hyperplanes, there always
exists two sets of points with the same projections on all of these
hyperplanes. Theorem I.2 from
\cite{reconstruct_projection} indicates that a necessary condition to ensure
uniqueness in this case is $\nproj > \npoints$.
In our Theorem
\ref{thm:as_unicity_reconstruction}, the points are fixed and
uniqueness of the reconstruction holds almost surely when the $P_i$
follow assumption $\mathcal{A}_\P$ and as soon as $D > d$, whatever
the number $\npoints$ of points in the discrete measure.

\subsection{Details on the critical case \texorpdfstring{$\sum_i d_i =d$}{sum di eq d}}\label{sec:details_critical_case}

In the theorem below, we show that the example in
\ref{fig:reconstruction_grid} is representative of the critical case.

\begin{theorem}[Number of admissible points in the critical case]\label{thm:support_critical_case}
	\

	Let $\gamma_Z$ be a fixed discrete probability measure. Assume that
  the matrices $P_i$ follow assumption $\mathcal{A}_\P$, and that  $D:=\Sum{i=1}{\nproj}d_i = d$. 
Then the cardinality of $S$ is exactly $\npoints^\nproj$, $\P$-a.s..
\end{theorem}

\begin{proof}
  
\blue{Using the notation $\Bell := (\ell_1, \cdots, \ell_\nproj) \in \llbracket 1, \npoints \rrbracket^\nproj$,} we know that $S = \Reu{\substack{\Bell \in \llbracket 1, \npoints \rrbracket^\nproj}}{}\quad S_{\Bell}$ where $ S_{\Bell} = \Inter{i=1}{\nproj}\;\left(z_{\ell_i} + \Ker P_i\right)$. 
Following the proof of Theorem~\ref{thm:as_unicity_reconstruction}, in the case $D = d$, we see that assumption $\mathcal{A}_\P$ implies that $ S_{\Bell} $ is almost surely a singleton $\{ x_{\Bell} \}$. It follows that $S$ is almost surely the union of at most $\npoints^\nproj$ singletons.
Let us show that if $\Bell \neq \Bell'$ then  $ S_{\Bell} \cap   S_{\Bell'} = \varnothing$ a.s..
Indeed, if $x$ belongs to  $ S_{\Bell} \cap   S_{\Bell'}$ then $x$ is solution of a linear system of $2d$ equations:

$$\forall i \in  \llbracket 1, \nproj\rrbracket,\; \forall j \in \llbracket 1, d_i\rrbracket,\;\;\;\left\lbrace\begin{array}{c}
    (u_i^{(j)})^T x = (u_i^{(j)})^T z_{\ell_i}  \\
    (u_i^{(j)})^T x = (u_i^{(j)})^T z_{\ell'_i} 
\end{array}\right. ,$$
which implies $\forall i \in  \llbracket 1, \nproj\rrbracket,\; \forall j \in  \llbracket 1, d_i\rrbracket,\;\;\; \ell_i = \ell'_i, \;\text{ or }\;  \ell_i \neq \ell_i'\;\text{and}\;u_{\ell_i}^{(j)} \in (z_{\ell_i }- z_{\ell'_i})^{\perp}$.

Now, under $\mathcal{A}_\P$, if $\ell_i \neq \ell_i'$, then $\mathbb{P}( u_{\ell_i}^{(j)} \in (z_{\ell_i }- z_{\ell'_i})^{\perp}) = 0$, and thus $\Bell = \Bell'$ a.s..

\end{proof}

Let us clarify what the set of solutions $\mathcal{S}$ looks like in  this critical case $D = d$. Let $\gamma$ be a solution of \ref{eqn:RP} and denote $S = (x_\Bell)_{\Bell \in \llbracket 1, \npoints \rrbracket^\nproj}$. By Theorem \ref{thm:support_critical_case}, $\gamma$ is of the form $\gamma = \Sum{\Bell \in \llbracket 1, \npoints \rrbracket^\nproj}{} a_\Bell \delta_{x_\Bell}$. %
Now, since $\gamma$ is a solution, we have for $i \in \llbracket 1, \nproj \rrbracket,\; P_i\#\gamma = P_i\#\gamma_Z$, thus$\Sum{\Bell \in \llbracket 1, \npoints \rrbracket^\nproj}{}a_\Bell \delta_{P_iz_{\ell_i}} = \Sum{k=1}{\npoints}b_k\delta_{P_iz_k}$. 
Since the $(P_iz_\ell)_\ell$ are all distinct a.s., this entails for all $k \in \llbracket 1, \npoints \rrbracket$: $\Sum{\Bell_{-i} \in \llbracket 1, \npoints \rrbracket^{\nproj-1}}{}a_{\ell_1, \cdots, \ell_{i-1}, k, \ell_{i+1}, \cdots, \ell_\nproj} = b_k$, where $\Bell_{-i}$ indicates that we index this $(\nproj-1)$-tuple on $\llbracket 1, \npoints \rrbracket \setminus \lbrace i \rbrace$.
We can re-write this condition as $a \in \Pi_\npoints^\nproj(b)$, the set of $\npoints$-dimensional $\nproj$-tensors on $\R_+$ ($\R_+^{\npoints^\nproj}$) with all $\nproj$ marginals equal to $b$.
Conversely, if $\gamma$ is of the form $\gamma = \Sum{\Bell \in \llbracket 1, \npoints \rrbracket^\nproj}{} a_\Bell \delta_{x_\Bell}$ with $a \in \Pi_\npoints^\nproj(b)$, then we have by construction $\forall i \in \llbracket 1, \nproj \rrbracket,\; P_i\#\gamma = P_i\#\gamma_Z$ and thus $\gamma$ is a solution.

{In the particular case where $\gamma_Z$ is uniform, if we restrain $\gamma$ to be also a uniform measure, the problem in this critical case has a combinatorial amount of solutions. Without this
restriction, the problem has an infinite amount of
solutions, as is discussed in the particular case of \ref{fig:reconstruction_grid}.}

\section{Consequence for the empirical Sliced Wasserstein Distance}

The Sliced Wasserstein  (SW) distance  between probability measures is frequently used
in applied fields such as image processing or machine learning, as an efficient alternative to the Wasserstein distance. It was introduced
in~\cite{Rabin_texture_mixing_sw} to generate barycenters between images of
textures, and it is commonly used nowadays as a loss
~\cite{karras2018progressive,deshpande_generative_sw,Wu2019_SWAE} to train
generative networks. This distance writes: $$\forall \alpha, \beta \in \mathcal{P}_2(\R^d),\quad\SW^2(\alpha,\beta) = \int_{\theta \in \SS^d}{} \W_2^2(P_{\theta} \# \alpha,P_{\theta} \# \beta) \dd\bbsigma(\theta),$$
where $\bbsigma$ is the uniform distribution over the unit sphere $\SS^d$ of $\R^d$, and $P_{\theta}$ denotes the linear projection on the line of direction $\theta$. In its empirical (Monte-Carlo) approximation, used for numerical applications, it becomes:

\begin{equation}\label{eq:SW_MC}
\forall \alpha, \beta \in \mathcal{P}_2(\R^d),\quad\widehat{\mathrm{SW}}_\nproj^2(\alpha, \beta) := \cfrac{1}{\nproj}\Sum{i=1}{\nproj}\mathrm{W_2^2}(P_{\theta_i} \#\alpha, P_{\theta_i}\#\beta).
\end{equation}

\blue{The main advantage of SW over the usual Wasserstein distance is
  computational: for two $d$-dimensional uniform discrete measures
  with $n$ samples, the empirical estimation with $\nproj$ projections
  \ref{eq:SW_MC} can be computed in $\mathcal{O}(\nproj dn + \nproj n \log(n))$ (\cite{nadjahi2021sliced}, \S 2.6), leveraging the fact that the 1D projected Wasserstein distances can be computed by a sorting algorithm. In the same setting, the Wasserstein distance can only be computed in super-quadratic complexity with respect to $n$ \cite{computational_ot}.}

Since $\W_2$ is a distance on $\mathcal{P}_2(\R^d)$ (the space of probability measures over $\R^d$ admitting a finite second-order moment), $\widehat{\mathrm{SW}}_\nproj$ is non-negative, homogeneous and satisfies the triangle inequality. However,  $\widehat{\mathrm{SW}}_\nproj$ is only  a pseudo-distance since it does not satisfy the separation property: whatever the $\nproj$ directions chosen, it is always possible to find two different distributions $\alpha$ and $\beta$ such that  $\widehat{\mathrm{SW}}_\nproj(\alpha, \beta) =0$. Now, our previous reconstruction results show that when the $\nproj$
directions are drawn from $\bbsigma$ and $\beta$ is a fixed discrete measure, then $\forall \alpha \in \mathcal{P}_2(\R^d),\; \widehat{\mathrm{SW}}_p^2(\alpha,\beta)~=~0~\Longrightarrow~\alpha~=~\beta$ almost surely provided that $\nproj > d$. Indeed,
$ \widehat{\mathrm{SW}}_\nproj(\alpha, \beta) =0$ if and only if $\alpha$
belongs to the set $\mathcal{S}$ (for $\gamma_Z = \beta$).
On the contrary, when the number of projections is too small, the set of discrete measures at distance $0$ from a given one is infinite, as stated in the theorem below.

\begin{theorem}[]\label{thm:SW_insufficient_projections}
	
	Let {$\gamma_Z := \Sum{\ell=1}{\npoints}b_\ell\delta_{z_\ell}$}, where the $(z_\ell)$ are fixed and distinct. Assume $\theta_1, \cdots, \theta_\nproj \sim \bbsigma^{\otimes \nproj}$.
	
	\begin{itemize} 
	\item if $\nproj \leq d$, there exists $\bbsigma$-a.s. an infinity of measures $\gamma \neq \gamma_Z \in \mathcal{P}_2(\R^d)$ s.t. $\widehat{\mathrm{SW}}_\nproj(\gamma, \gamma_Z) = 0$.

	\item if $\nproj > d$, we have $\bbsigma$-almost surely $\lbrace\gamma_Z\rbrace = \underset{\gamma \in \mathcal{P}_2(\R^d)}{\argmin}\; \widehat{\mathrm{SW}}_\nproj(\gamma, \gamma_Z)$.
	
	\end{itemize}
\end{theorem}

In the limit case $\nproj = d$, the distance can be grown by scaling the
points of $\gamma_Z$ further away from the origin. In the case $\nproj < d$, the supports of solution measures can be infinitely far from the support of $\gamma_Z$, as illustrated in \ref{fig:linear_push_forward}. \blue{To conclude, if $\nproj\leq d$, the information $\widehat{\mathrm{SW}}_\nproj(\gamma, \gamma_Z) \approx 0$ yields no valuable information regarding the closeness of $\gamma$ and $\gamma_Z$. If $\nproj>d$, the information $\widehat{\mathrm{SW}}_\nproj(\gamma, \gamma_Z) = 0$ yields $\gamma=\gamma_Z$ almost-surely with random projections. Due to possible local optima of $Y\longmapsto \widehat{\mathrm{SW}}_\nproj(\gamma_Y, \gamma_Z)$, knowing only $\widehat{\mathrm{SW}}_\nproj(\gamma, \gamma_Z) \approx 0$ may be too weak to conclude that the measures are close even in the favorable case $\nproj > d$, although these considerations are beyond the scope of this work.}

\section{Conclusion: Discussion on SW as a Loss in Machine Learning}

In Sliced-Wasserstein-based Machine Learning, the question of global optima is
paramount since in practice, one must default to a surrogate of $\SW$: be it
through stochastic gradient descent (drawing a batch of $\theta_i$ at each
iteration), or directly through the estimation $\widehat{\mathrm{SW}}_\nproj$. %
To be precise, generative models such as
\cite{deshpande_generative_sw} minimize $\theta \longmapsto
\mathrm{SW}(T_\theta\#\mu_0, \mu)$ - or a surrogate thereof - where $\mu_0$ is a
low-dimensional input distribution (often chosen as realizations of Gaussian
noise), where $\mu$ is the target distribution (the discrete dataset), and where
$T_\theta$ is the model of parameters $\theta$. In this case, the dimension $d$
of the data, which for images can easily exceed one million, can be
very large. \blue{Our results suggest that performing optimisation with less
than $d$ projections is unsound in this case, since it
leads to solutions that can be arbitrarily far away from the true data
distribution. } \blue{However, these results do not take into account
the  intrinsic dimension of the target measure $\gamma_Z$, which in
the case of images is probably supported on a $d'$-dimensional
manifold with $d'\ll d$. This question has been addressed very recently
in~\cite{tachella2023sensing}, which was unknown to us at the time of
working on this paper, and provides insightful reconstruction results
when $\gamma_Z$ is known to belong to a $d'$-dimensional manifold.}

Furthermore, it is important to underline that having the
guarantee that the global optima are the desired original measure is
insufficient in practice. Indeed, the landscape $Y \longmapsto
\widehat{\mathrm{SW}}_\nproj(\mu_Y, \mu_Z)$ can present numerous local optima, which
can limit the usefulness of SW as a loss function.
For practical considerations, this study on global optima could be complemented by an analysis of the aforementioned landscape, which we leave for future work. \blue{Namely, one may find a study of the optimization properties of $Y \longmapsto
\widehat{\mathrm{SW}}_\nproj(\mu_Y, \mu_Z)$ and $Y \longmapsto
\mathrm{SW}(\mu_Y, \mu_Z)$ in \cite{discrete_sliced_loss}, and a related extension to the study of Stochastic Gradient Descent for SW generative networks in \cite{nn_loss}.}

\newpage
\bibliography{ecl}
\bibliographystyle{plain}
\end{document}